\documentclass[12pt,leqno]{amsart}
\usepackage{amsmath,amsfonts,latexsym,graphicx,amssymb,url, color}
\usepackage{txfonts} 

\usepackage[hmargin=2.5 cm,vmargin=2.0 cm]{geometry}

\usepackage{graphicx}
\usepackage{color} 

\newcommand{\R}{{\mathbb R}} 

\newcommand{\N}{{\mathbb N}}
\newcommand{\D}{\mathcal{D}}

\newcommand{\oL}{\mathcal L}

\newtheorem{theorem}{Theorem}[section]
\newtheorem{definition}[theorem]{Definition}
\newtheorem{remark}[theorem]{Remark}
\newtheorem{lemma}[theorem]{Lemma}
\newtheorem{proposition}[theorem]{Proposition}

\numberwithin{equation}{section}

\newcommand{\beq}{\begin{equation}}
\newcommand{\eeq}{\end{equation}}

\definecolor{darkred}{rgb}{.70,.12,.20}

\definecolor{darkgreen}{rgb}{.20,.52,.14}

\title[Mixed Boundary Value Problems]{ Mixed Boundary Value Problems for non-Divergence Type \\  Elliptic Equations in Unbounded Domain }
\author[D. Cao]{Dat Cao$^{\dag}$}
\email{dat.cao@ttu.edu}
\author[A. Ibraguimov]{Akif Ibraguimov$^{\dag}$}
\email{akif.ibraguimov@ttu.edu}
\address{$^{\dag}$ Department of Mathematics and Statistics, Texas Tech University, Box 41042, Lubbock, TX 79409-1042}
\author[A. I. Nazarov]{Alexander I. Nazarov$^{\ddag}$}
\email{al.il.nazarov@gmail.com}
\address{$^{\ddag}$ St.Petersburg Department of Steklov Institute, Fontanka 27, St.Petersburg, 191023, Russia, 
and St.Petersburg State University, Universitetskii pr. 28, St.Petersburg, 198504, Russia}
\begin{document}

\begin{abstract} 

We investigate the qualitative properties of solution to the Zaremba type problem in unbounded  domain for the non-divergence elliptic equation with possible 
degeneration  at infinity. 
The main result is Phragm\'en-Lindel\"of type principle on growth/decay of a solution at infinity depending on both the structure of the Neumann portion of the boundary
and the "thickness" of its Dirichlet portion.  The result is formulated in terms of so-called $s$-capacity 
of the Dirichlet portion of the boundary, while the Neumann boundary should satisfy certain ``admissibility'' condition in the sequence of layers converging to 
infinity.  
 
\end{abstract}

\maketitle

\section{introduction}

We consider  non-divergence type elliptic operator
\begin{equation}\label{equ}
{\oL }u := - \sum_{i,j=1}^{n}a_{i j}(x) D_iD_ju \qquad\mbox{in}\quad \D.
\end{equation}
Such operators arise in theory of stochastic processes and other  applications, see, e.g., \cite{Krylov1,Dynkin,Fridman}.

In (\ref{equ}) $\D$ is an unbounded domain in $\mathbb{R}^n$, $n\geq 3$, and $D_i$ stands for the differentiation with respect to $x_i$.

We suppose that the boundary $\partial \D$ is split $\partial \D=\Gamma_1\cup \Gamma_2$. Here $\Gamma_1$ is support of the Dirichlet condition, 
and $\Gamma_2$ is support of the oblique derivative condition:
\begin{equation}\label{BC}
u(x)=\Phi(x) \ \ \mbox{on} \ \ \Gamma_1;\quad
\frac{\partial u}{\partial \ell}(x):=\lim_{\varepsilon \to +0} \frac {u(x)-u(x-\varepsilon \ell)}{\varepsilon}=\Psi(x) \ \ \mbox{on} \ \ \Gamma_2,
\end{equation}
where $\ell=\ell(x)$ is a measurable, uniformly non-tangential outward vector field on $\Gamma_2$. Without loss of generality we can suppose $|\ell|\equiv1$. 
We call $\Gamma_1$ the Dirichlet boundary, and $\Gamma_2$ the Neumann boundary.

We say  that $\D$ is $x_1$-oriented if there exists an increasing sequence  $\tau_j \to +\infty$ such that
 the cross-section
$$
\D(j)=\{x \in \D: x_1=\tau_j\} 
$$ 
is non-empty and connected for all $j$. 
Denote 
$$
\D(k,j)=\{x \in \D: \tau_k < x_1 < \tau_{j}\},\qquad k<j.
$$

We  prove Phragm\'en-Lindel\"of type theorem in unbounded domain. Roughly speaking it states that if the Wiener type series diverges then a positive sub-elliptic 
function, which vanishes on the Dirichlet boundary in a neighborhood of infinity tends either to zero or to infinity with prescribed speed as $x_1\to \infty$.    
Corresponding result in bounded domains for non-degenerate equation was obtained in \cite{AI-Naz}.

For divergence type equation in case of pure Dirichlet problem the result of this type was first proved in very general case by Maz'ya in terms of Wiener 
capacity in \cite{MazyaFL1}. Criteria for regularity at infinity in cylindrical type domain  
for Zaremba problem was obtained by Maz'ya and co-athors in \cite{MazyaFL}.

Here we consider the case of non-divergence  equation which may degenerate, and domain which may narrow or widen at infinity.
The Neumann boundary $\Gamma_2$ is supposed to satisfy ``inner cone'' condition, see \cite{nad-max-principle, nad-max-principle1}, and the ``admissibility'' 
condition, see Definition \ref{adm} below. In the case  of pure Dirichlet problem 
($\Gamma_2=\emptyset$) similar questions for certain class of domains were discussed by E.M. Landis \cite{landis-book,landis-paper},  
A.T. Abbasov \cite{abb}, and sharpened by Yu.A. Alkhutov \cite{Alkhutov}. \medskip

We always assume that the matrix of  coefficients is measurable and symmetric, and satisfies the ellipticity condition in any finite layer $\D(j-1,j+1):$
\begin{equation}\label{e1} 
e(j):=\max_{|\xi|=1}\sup_{x\in \D(j-1,j+1)} e(x,\xi)<\infty,
\end{equation}
where  $e(x,\xi)$ is the ellipticity function (see \cite{landis-book, Alkhutov})
\begin{equation}\label{ ellipticity function}
e(x,\xi)=  \frac{\sum _{i=1}^{n} a_{ii}(x)}{\sum_{i,j=1}^{n}a_{i j}(x)\xi_i\xi_j}.
\end{equation}
However, we admit that $e(j)$ may grow to $\infty$ as $j\to \infty$.

We will investigate coupled  impact of the degeneracy of the equation and speed of the domain narrowing on the behavior of solutions 
of mixed BVP. \medskip

The paper is organized as follows.
In Section \ref{pre} we review some known results about non-divergence equations:
the boundary point lemma in the form of Nadirashvili, 
the Landis Growth Lemma in case $\Gamma_2 = \emptyset$, and Growth Lemma in Krylov's form.

The Growth Lemma,  first introduced by Landis in \cite{Landis2, Landis1}, is a fundamental tool to study 
qualitative properties and regularity of solutions in bounded and unbounded domain. Recent review on Growth Lemma and its applications was given in 
\cite{safonov-1} (see also \cite{Aimar}).

In Section \ref{growth-admissible} we introduce domains with admissible Neumann boundary and  prove Growth Lemma for that type of domains.

In the last Section \ref{dichotomy}, dichotomy theorem is proved for solutions of mixed boundary value problem.\medskip

We always assume that $u \in W^2_{n, loc}(\D)\cap {\mathcal C}^{1}_{loc}(\D\cup\Gamma_2)\cap {\mathcal C}(\overline\D)$.

$B(x,R)$ stands for the ball with radius $R$ centered in $x$. By $C$ we denote any absolute constant.

\section{ Preliminairy results} \label{pre}

Recall that a function $u$ is called super-elliptic (resp. sub-elliptic) in $\D$  if ${\oL} u \geq 0$ (resp. ${\oL}u\leq 0$) in $\D$.

First we formulate a corollary of classical Aleksandrov-Bakel'man maximum principle, see, e.g., \cite{Ale} or survey \cite{naz-survey}.
\begin{proposition}\label{maxprinc}
Let $\D$ be a bounded domain. Let $u$ be super-elliptic (sub-elliptic) in $\D$. Then for $x\in\overline{\D}$
$$
u(x) \ge \min_{\partial \D} u \qquad \big({\rm resp.} \ \ u(x) \le \max_{\partial \D} u\big).
$$
 \end{proposition}

We say that $\Gamma_2$ satisfies {\bf inner cone condition} if there are $0<\varphi<\pi/2$ and $h>0$ such that for any $y\in \Gamma_2$ there exists a 
right cone $K\subset \D$ with the apex at $y$, apex angle $\varphi$ and the height $h$.

In \cite{nad-max-principle}, \cite{nad-max-principle1} N. Nadirashvili obtained the following fundamental generalization of the
Hopf-Oleinik boundary point lemma\footnote{A historical survey of this result can be found in \cite{Ap-Naz}.}. Notice that original papers of Nadirashvili deal with 
$u\in\mathcal{C}^2(\D)\cap {\mathcal C}^{1}_{loc}(\D\cup\Gamma_2)$ but using the the Aleksandrov-Bakel'man maximum principle one can transfer 
the proof without changes for $u \in W^2_{n, loc}(\D)\cap {\mathcal C}^{1}_{loc}(\D\cup\Gamma_2)$.

\begin{proposition}\label{nad} Let $\D$ be a bounded domain, and let a non-constant function $u$ be super-elliptic (sub-elliptic) in $\D$. Suppose that $y\in\Gamma_2$ 
and $u(y)\le u(x)$  (resp. $u(y)\ge u(x)$) for all $x \in \D$.  Let $\Gamma_2$ satisfy inner cone condition in a neighborhood of $y$.
Then for any  neighborhood $S$  of $y$ on $\Gamma_2$ and for any $\epsilon<\varphi$ there exists a point $\tilde x \in S$ s.t.
\begin{equation}
\frac{\partial u}{\partial \ell}(\tilde x)<0 \qquad \Big({\rm resp.} \ \ \frac{\partial u}{\partial \ell}(\tilde x)>0\Big) 
\end{equation}
for any direction $\ell$ s.t. the angle $\gamma$ between $\ell$ and the axis of $K$ is not greater then $\varphi-\epsilon$, see Fig. \ref{Cone}.   
\end{proposition}

\begin{figure}[ht]
\begin{center}
\advance\leftskip-3cm
\advance\rightskip-3cm
\includegraphics[scale=0.5]{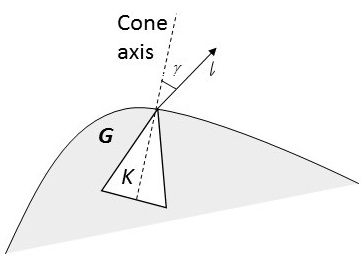}
\caption{Inner cone condition}.
\label{Cone}
\end{center}\end{figure}

From Propositions \ref{maxprinc} and \ref{nad} we obtain the comparison theorem for mixed boundary value problem.

\begin{proposition}\label{comparison_theorem}
Let $\D$ be a bounded domain, and let $\Gamma_2$ satisfy inner cone condition. Assume the vector field $\ell$ on $\Gamma_2$ satisfies the same condition as in 
Proposition \ref{nad}.

Suppose that $u,v \in W^2_{n, loc}(\D)\cap {\mathcal C}^{1}(\D\cup\Gamma_2)\cap {\mathcal C}(\overline \D)$.
If ${\oL} u \leq {\oL} v  $ in $\D$, $u\leq v $ on $\Gamma_1$, and
$\frac{\partial u}{\partial \ell}\leq \frac{\partial v}{\partial \ell}$ on $\Gamma_2$  then $v\geq u $ in $\overline \D$.
\end{proposition}

We recall the well-known notion of $s$-capacity, see, e.g., \cite{landis-book, landkof}. 
\begin{definition}\label{s-capacity}
Let $s>0$ and let $H$ be a Borel set in $\R^n$. Let a measure $\mu$ be defined on Borel subsets of $H$. We 
write $\mu \in \mathcal{M}(H)$ if 
\begin{equation}
\int\limits_{H} \frac{d\mu(y)}{|x-y|^s} \leq 1, \quad  \text{for} \quad x\in \R^n\setminus H. 
\end{equation}  
Then the quantity
\begin{equation}
{ \mathbf{C}}_s(H)=\sup_{\mu\in \mathcal{M}(H)} \mu(H) 
\end{equation}
is called {\bf $s$-capacity} of $H$.
\end{definition}

We also recall the following simple statement, see  \cite[Ch. I, Lemma 3.1]{landis-book}.

\begin{proposition}\label{subsol}
Let 
\begin{equation*}\label{e} 
e:=\max_{|\xi|=1}\sup_{x\in \D }e(x,\xi) < \infty.
\end{equation*}. 
If $s\ge e-2 $ then $ \oL |x|^{-s}\le 0$ for $x\ne0$. 
\end{proposition}

Next, we give a quantitative variant of the Landis Growth Lemma (see \cite[Ch. I, Lemma 4.1]{landis-book}).

\begin{proposition} \label{Land}
 Let $u$ be sub-elliptic in $\D \cap B(x^0,aR)$. Suppose that $u>0$ in $\D$ and $u=0$ on 
 $\Gamma_1=\partial \D \cap B(x^0,aR)$. Then 
  
\begin{equation}\label{boundlayer_capacity}
\sup_{\D \cap B(x^0,a R)} u\ge \frac{\sup_{\D \cap B(x^0,\frac a4 R)} u}{ 1-\eta_1 \mathbf{C}_s(H)R^{-s}},
\end{equation} 
where $s=e-2$, 
\begin{equation} \label{eta-1}
\eta_1=\eta_1(s)=\left(\frac{2}{a}\right)^s\left[1-\left(\frac{2}{3}\right)^s\right],
\end{equation}
and $H=B(x^0,\frac a4 R)\setminus \D$.

Consequently if there exists a ball $B_{q R}$ of radius of  $q R \ (0<q<\frac a4)$  belonging to $H$ then  
\begin{equation}\label{etazero}
\sup_{\D \cap B(x^0,a R)} u\ge \frac{\sup_{\D \cap B(x^0, \frac a 4R)} u}{1-\eta_1q^s}.
\end{equation} 
%
\end{proposition}

The following definition of barrier and Growth Lemma for mixed boundary problem was introduced in \cite[Sec. 2]{AI-Naz}. For the reader's convenience
we give it with the full proof. For the Dirichlet boundary conditions this type of Growth Lemma was first introduced in \cite{krylovizvest}.   

\begin{definition}\label{strict_growth}
Let $\D$ be a  domain with  boundary $\partial \D=\Gamma_1\cup \Gamma_2$.  Assume that ``small''   ball $B(x^0,R)$ and ``big'' ball $B(x^0,aR)$, $a>1$, are 
intersecting with $\D$. 

We call the function $w$ {\bf barrier} for the mixed boundary value problem in the balls $B(x^0,R)$ and $B(x^0,aR)$ if the following conditions hold:
\begin{equation}\label{Lw<0}
 w \ \text{is sub-elliptic } \  \text{in} \ \D \cap B(x^0,aR).
\end{equation}
\begin{equation}\label{w<1}
w(x)\leq 1 \ \text{on} \ \Gamma_1\cap B(x^0,aR).
\end{equation} 
\begin{equation}\label{w_l<0}
\frac{\partial w}{\partial \ell}\leq 0 \ \text{on} \ \Gamma_2\cap B(x^0,aR).
\end{equation} 
\begin{equation}\label{w<0}
w\leq 0 \ \text{on} \ \overline \D \cap \partial B(x^0,aR).
\end{equation}  
\begin{equation}\label{w>eta}
w(x)\ge \beta_0 \ \text{in } \
 B(x^0,R)\cap \D
\end{equation}   
for some $\beta_0>0$.
\end{definition} 

\begin{lemma}\label{growth-one}
Let a domain $\D$ and let balls $B(x^0,R)$, $B(x^0,aR)$ be the same as in  Definition \ref{strict_growth}.
Suppose that $ \Gamma_2$  satisfies the inner cone condition and the vector field $\ell$ on $\Gamma_2$ satisfies the same condition as in Proposition \ref{nad}. 
Assume that there exists a barrier $w$ for mixed BVP in the balls $B(x^0,R)$ and $B(x^0,aR)$.

Suppose that a function $u$ is sub-elliptic in $\D \cap B(x^0,aR)$, $u>0$ in $\D$, $u\leq 0$ on $\Gamma_1\cap B(x^0,aR)$ and 
$\frac{\partial u}{\partial \ell}\leq 0$ on $\Gamma_2\cap B(x^0,aR)$. Then 
\begin{equation}\label{grow-prop-1}
\sup_{\D \cap B(x^0,a R)} u\ge \frac{\sup_{\D \cap B(x^0,R)} u}{1-\beta_0},
\end{equation}  
where $\beta_0$ is the constant from Definition \ref{strict_growth}.
\end{lemma}

\begin{proof}
Let $M=\sup_{\D \cap B(x^0,aR)} u$. We define
\begin{equation}
v(x)=M(1-w(x)).
\end{equation} 
Obviously ${\oL} v \geq 0\geq {\oL} u,$  $ v(x)\geq 0 \geq u(x)$ on $\Gamma_1\cap B(x^0,aR)$,  
$ \frac{\partial v}{\partial \ell} \ge 0\ge \frac{\partial v}{\partial \ell}$ on $\Gamma_2$, and $v(x)\geq M \geq u(x)$ on $\partial B(x^0,aR)\cap \D$. 
Applying Proposition \ref{comparison_theorem} to functions $v(x)$ and $u(x)$ in the domain $\D \cap B(x^0,aR)$ we get that $v(x) \geq u(x)$. This gives 
in the intersection $\D \cap B(x^0,R)$, with regard of (\ref{w>eta}),
\begin{equation*}
M(1-\beta_0)\geq M(1-\inf_{\D \cap B(x^0,R)} w)\geq  \sup_{\D \cap B(x^0,R)} u, 
\end{equation*}
and the statement follows.
\end{proof} 

Let us introduce a sufficient condition for existence of the barrier in the Definition \ref{strict_growth}.

\begin{lemma} \label{barrier-2}
Let $\D$ and balls $B(x^0,R)$, $B(x^0,aR)$ be the same as in the Definition \ref{strict_growth}. 

Assume that there exists a ball  $B(x^0,\alpha R)$, $0<\alpha<1$, such that $B(x^0,\alpha R)\cap\D=\emptyset$ and $\Gamma_1$ separates the ball 
$B(x^0,\alpha R)$ from $\Gamma_2 $ in $ B(x^0,aR)$. 
Assume that  \begin{equation}\label{(x-x^0)dotl}
(x-x^0)\cdot \ell \ge0\end{equation} 
for any  $x\in\Gamma_2\cap B(x^0,aR)$.
 
Then for $s\ge e-2$ the function $\displaystyle w(x)= \alpha^s \big(R^s|x-x^0|^{-s} - a^{-s}\big)$ is a barrier for the mixed BVP in the balls  $B(x^0,R)$ and 
$B(x^0,aR)$, with $\beta_0:=\alpha^s(1 - a^{-s})$.
\end{lemma}

\begin{proof} Let us check each conditions in the Definition \ref{strict_growth}:
\
\begin{enumerate}

\item condition (\ref{Lw<0}) follows by Proposition \ref{subsol} ;

\item on $\Gamma_1\cap B(x^0,aR)$, we have $|x-x^0| > \alpha R$, so 
$$
w(x) \le \alpha^s R^s(\alpha R)^{-s} -\frac{\alpha^s}{a^s} = 1- \frac{\alpha^s}{a^s} \le 1,
$$ 
and (\ref{w<1}) follows;

\item   on $ \overline{\D}  \cap \partial B(x^0,aR)$ we have $w(x)= \alpha^s R^s(aR)^{-s} - \frac{\alpha^s}{a^s}=0$, and (\ref{w<0}) follows;

\item by condition \eqref{(x-x^0)dotl} we have 
$$
\frac{\partial w}{ \partial \ell } = -s \alpha^sR^s \frac{(x-x^0) \cdot  \ell }{|x-z|^{s+2}} \le 0,
$$
and (\ref{w_l<0}) follows;

\item
for $|x-x^0| \le R$ we have
\begin{equation} \label{beta-0}
w(x) \ge \alpha^s R^sR^{-s} - \frac{\alpha^s}{a^s}=\alpha^s(1 - \frac{1}{a^s}) >0,
\end{equation}
and (\ref{w>eta}) follows.
\end{enumerate}
\end{proof}

\section{ Growth lemma for admissible  domains} \label{growth-admissible}

\begin{definition}\label{adm}
Let unbounded domain $\tilde\D$ in $\R^n$ be $x_1$-oriented, and let $\tau_j \to +\infty$ be corresponding increasing sequence.

Suppose that  $\Gamma_2=\partial \tilde\D$  satisfies the inner cone condition and the vector field $\ell$ on $\Gamma_2$ satisfies the same condition 
as in Proposition \ref{nad}.   

We say that $\tilde\D$  is  {\bf admissible} w.r.t.  $\ell$ if there exist $\rho\le1$, $a>1$, $N_0\in\N$ s.t. for all sufficiently large $j\in\N$ the 
following conditions hold: \medskip
 
Set 
$R_j:=  \rho(\tau_{j+1}-\tau_{j})$. \medskip

$A)$ There is a point $z^0 \in \tilde\D(j)$ s.t. $B(z^0,aR_j) \cap \Gamma_2 = \emptyset$.\medskip

For any point $\xi\in\tilde\D(j)$, there exists a finite sequence of points $ z^i \in \tilde\D(j)$,
$i=1,\dots,N$, $N\le N_0$, such that\medskip

$B)$ $\xi \in B(z^N,R_j)$, and the intersection $B(z^i,R_j) \cap B(z^{i+1},R_j)\cap \tilde\D$ contains a ball ${ B}_{qR_j}$ of radius $ qR_j$, for some $0<q<a/4$;\medskip

$C)$ for any $x \in B(z^i,aR_j)\cap  \Gamma_2$,  $i=1,\dots, N$, we have $(x-z^{i})\cdot \ell \geq 0$.
\end{definition}
Without loss of generality we assume forth-forward that $a<4$ in above.

\begin{remark}

Conditions in admissibility relate only to the Neumann boundary of the domain, and can be interpreted as follows:

 (A): domain is  wide enough; 
 
 (B): domain is not too wide; 
 
 (C): boundary $\Gamma_2$  is regular.

\end{remark}

It is easy to see that if $\tilde\D$ is a convex cylinder with smooth boundary, and $\ell={\bf n}(x)$ is outward normal to $\Gamma_2=\partial\tilde\D$ 
at the point $x$, then $\tilde\D$ is admissible for $\tau_j=j$. Also, if $\tilde\D$ is a convex acute cone with smooth boundary, and $\ell={\bf n}(x)$, 
then $\tilde\D$ is admissible for $\tau_j=2^j$. 
 
We consider a generalization of these examples. 

Let $\Omega\subset \mathbb R^{n-1}$ is convex (not necessary smooth), and
\begin{equation}\label{nardom}
\tilde\D=\{x=(x_1,\bar x  f(x_1)), \quad \bar x \in \Omega,\ x_1>1 \},
\end{equation}
where $f$ is a positive function s.t.
\begin{equation}\label{subexp}
\lim_{x_1\to \infty} f'(x_1)=0. 
\end{equation}
Assume that exists  of monotonically increasing sequence  
 $\tau_j \to \infty$ s.t. 
\begin{equation}\label{feqvR}
 \ C^{-1}R_j< f(x_1)<CR_j \quad \text{for} \ x_1\in[\tau_{j-1},\tau_{j+1}], \ \text{ uniformly in    } j\in\N. 
\end{equation}

{\bf Example 1}. Let $f$ be regularly varying at infinity with index $\alpha<1$, see \cite{seneta}. This means
\begin{equation}\label{RVF}
f(t)=t^{\alpha}\phi(t),\quad \alpha<1, \quad \frac {t\phi'(t)}{\phi(t)}\to0\ \ \text{as}\ \ t\to+\infty.
\end{equation}
Set $\tau_j=F(j)$ where $F$ is inverse function to $\frac t{f(t)}$ (by (\ref{RVF}) $\frac t{f(t)}$ increases for sufficiently large $t$).
We claim that conditions \eqref{subexp}-\eqref{feqvR} are satisfied. 

Indeed, (\ref{subexp}) evidently follows from (\ref{RVF}). Further, for $j \ge 2$ we have
\begin{equation}\label{R1} 
R_j=\rho(\tau_{j+1}-\tau_j)= \rho\int\limits_j^{j+1}F'(\theta)\,d\theta=\rho\int\limits_j^{j+1}\frac {d\theta}{\big(\frac t{f(t)}\big)'\big|_{t=F(\theta)}}.
\end{equation} 
 By (\ref{RVF}) we derive
 $$
 \Big(\frac t{f(t)}\Big)'=\frac 1{f(t)}\cdot\Big[1-\alpha-\frac {t\phi'(t)}{\phi(t)}\Big]\sim\frac {1-\alpha}{f(t)} \quad \text{as}\ \ t\to+\infty.
 $$
Thus (\ref{R1}) gives for sufficiently large $j$
$$
R_j\sim \frac {\rho}{1-\alpha}\cdot f(F(\tilde\theta)),\qquad \tilde\theta\in[j,j+1],
$$
 and the claim follows.\medskip

\begin{lemma}\label{convex_prop}
Let $\Omega\subset \mathbb R^{n-1}$ be convex and let $\tilde\D$ be as in \eqref{nardom}. Suppose conditions \eqref{subexp}-\eqref{feqvR} are satisfied. 
Let  
$$
\ell(x)={\bf n}(x)\equiv\frac{(-f'(x_1),{\bf n}_{\bar{x}})}{\sqrt{1+f'(x_1)^2}} \qquad \text{for}\quad x=(x_1,\bar x  f(x_1)), \ \bar x \in \partial\Omega,
$$
where ${\bf n}_{\bar x}$ is outward normal to a supporting plane to $\Omega$ at $\bar x$.

 Then domain $\tilde\D$ is admissible. 
\end{lemma}

\begin{proof}
Properties (A-B) follow from \eqref{nardom} and \eqref{feqvR}. To verify (C) one can assume without loss of generality that for 
$z^i\in\tilde\D(j)$ we have ${\rm dist}({\bar z}^i,\partial \Omega)\ge \delta_0$ for some $\delta_0$ depending only on $\Omega$. 
Then for $x\in B(z^i,aR_j)\cap \partial \tilde\D$ we have
\begin{eqnarray*}
(x-z^i)\cdot \ell\sqrt{1+f'(x_1)^2}&=&-(x_1-z_1^i)f'(x_1)+(\bar{x}f(x_1)-\bar{z}^if(\tau_j))\cdot {\bf n}_{\bar{x}}\\
&=&-(x_1-z_1^i)f'(x_1)+(f(x_1)-f(\tau_j))\bar{z}^i\cdot {\bf n}_{\bar{x}}+f(x_1)(\bar{x}-\bar{z}^i)\cdot {\bf n}_{\bar{x}}.
\end{eqnarray*}
The last term can be estimated from below by convexity of $\Omega$, and we obtain for some $\delta_1>0$
\begin{equation*}
(x-z^i)\cdot \ell\sqrt{1+f'(x_1)^2}\geq f(x_1)\delta_0\delta_1-|{\bar z}^i|\max_{\tau_{j-1}\leq t \leq \tau_{j+1}}|f'(t)|\,CR_j
-\max_{\tau_{j-1}\leq t \leq \tau_{j+1}}|f'(t)|\,aR_j.
\end{equation*}
Finally, due to \eqref{feqvR} and \eqref{subexp} we have for $j$ large enough
\begin{equation*}
(x-z^i)\cdot \ell\geq \frac {f(x_1)}{\sqrt{1+f'(x_1)^2}}\cdot\Big[\delta_0\delta_1-C\cdot({\rm diam}(\Omega)+a)\max_{\tau_{j-1}\leq t \leq \tau_{j+1}}|f'(t)|\Big]>0,
\end{equation*}
and the Lemma follows.
\end{proof}

\begin{definition}\label{tilde_D}
Let $\tilde \D\subset \R^n$ be unbounded domain admissible w.r.t. vector field $\ell$, and $\Gamma_2=\partial \tilde \D$. Let $G\subset \tilde \D$ be a closed set.
Denote by $\D$ the connected component of $\tilde \D\setminus G$ adjacent to $\Gamma_2$ and put $\Gamma_1=\partial \D\setminus\Gamma_2$. 
Clearly $\Gamma_1$ and $\Gamma_2$ can have common points only at infinity. Such domains $\D$ will be also called {\bf admissible}.\medskip
\end{definition}

The following Growth lemma for the mixed BVP in $\D$ is the key ingredient in the proof of the Phragm\'en-Lindel\"of dichotomy.

\begin{lemma}\label{mixed-lm-2}
Assume as in Definition \ref{tilde_D} that $\D\subset \tilde{\D}$ is admissible w.r.t. vector field $\ell$. Let $s_j= e(j)-2$ and let for some $\varkappa_j>0$
\begin{equation}\label{cap-cond}
 \mathbf{C}_{s_j}(  B(z^0,R_j) \setminus\D) \ge \varkappa_j \mathbf{C}_{s_j}(\tilde{\D}(j-1,j+1)\setminus\D)
\end{equation}
(recall that balls $B(z^i,R_j)$ are introduced in Definition \ref{adm}, and ${\bf C}_s$ stands for the $s$-capacity).

Suppose that 
\begin{equation}\label{subell}
\oL u \le 0\quad \text{in}\ \ \D,\qquad u>0\quad  \text{in}\ \  \D,\qquad u \le 0 \quad \text{on}\ \ \Gamma_1, 
\qquad \frac{\partial u}{\partial \ell} \le 0 \quad \text{on}\ \ \Gamma_2. 
\end{equation}
Then for all sufficiently large $j\in\N$
\begin{equation}\label{M>m}
\sup_{ \D(j-1,j+1)} u \ge  \frac{\sup_{ \D(j)} u}{1 - \kappa_j \mathbf{C}_{s_j}(H_j)R_j^{-s_j}}.
\end{equation}
Here $H_j=\tilde{\D}(j-1,j+1)\setminus\D$,
\begin{equation}\label{kapaj}
\kappa_j=  \frac{C}{2^{N_0}}  \left(\frac{2q^{N_0}}{a}\right)^{s_j}\varkappa_j,
\end{equation}
$a$, $q$ and $N_0$ are the constants from Definition \ref{adm} . 
\end{lemma}

\begin{figure}[ht]
\begin{center}
\advance\leftskip-3cm
\advance\rightskip-3cm
\includegraphics[scale=0.5]{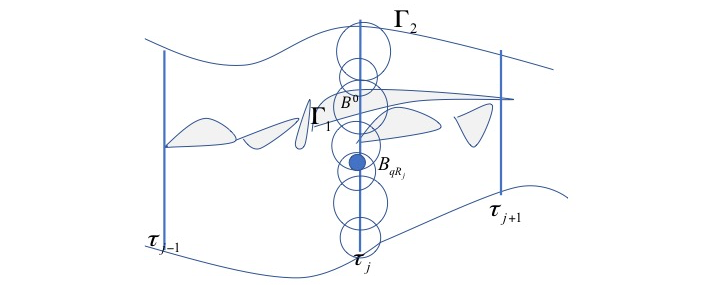}
\caption{Illustration to the Lemma \ref{mixed-lm-2}}
\label{Growthlemalayer}
\end{center}\end{figure}

\begin{proof} Let $j$ be so large that the assumptions (A-C) from Definition \ref{adm} are satisfied. For the sake of brevity we put
$\Omega=\D(j-1,j+1)$, $s=s_j$, $R=R_j$, $B^i=B(z^i,R_j)$, $\varkappa=\varkappa_j$, $H=H_j$, and 
$$
M=\sup_{\Omega } u , \qquad m=\sup_{\D(j)} u=u(\xi)
$$ 
(here $\xi\in \overline \D(j)$). 

We proceed similarly to the proof of \cite[Lemma 4.2]{AI-Naz}.
By assumption (B), $\xi\in B^N$ for some $N\le N_0$.
Consider the ball $B^0$ and the ball  $B(z_0,aR)$, $a>1$, concentric to it. Due to assumptions  (A) and (\ref{cap-cond}) 
we can apply  Lemma \ref{Land} to get  
\begin{equation*}\label{B1}
M:=\sup_{\Omega} u \geq \sup_{\Omega\cap B(z_0,a R)}u \geq \frac{\sup_{B^0\cap \Omega} u}{1-\varkappa\eta_1 {\bf C}_s(H)R^{-s}} 
\end{equation*}  
($\eta_1$ is defined in (\ref{eta-1})).

Suppose that
\begin{equation} \label{zeta0}
 \sup_{\Omega\cap B^0} u \geq m(1-\zeta_0), \qquad \text{where} \quad \zeta_0 = \frac{\varkappa\eta_1 {\bf C}_s(H)R^{-s}}{2-\varkappa\eta_1 {\bf C}_s(H)R^{-s}}. 
\end{equation}
Then we get 
\begin{equation*}\label{B11}
M \geq  \frac{m}{1-\varkappa\eta_2 {\bf C}_s(H)R^{-s}}
\end{equation*} 
for $\eta_2=\frac{\eta_1}{2}$,  and the statement follows. 

If \eqref{zeta0} does not hold, we consider the function 
\begin{equation}\label{func_v1}
u_1(x)=u(x)-m(1-\zeta_0),\qquad u_1(x)\leq 0 \quad \text{in}\ \ \Omega \cap B^0. 
\end{equation}

Let $\Omega_1:=\{ x: u_1 (x)>0\}$.  Assume that $B^1\cap \Omega_1\not= \emptyset$, otherwise we consider the first ball $B^i$ for which 
this property holds.

Suppose that 
\begin{equation}\label{tau}
 \sup_{B^1\cap \Omega} u_1 \geq m\zeta_0(1-\tau), 
\end{equation}
where the constant $\tau$ will be chosen later.
Consider any simply connected component of the domain $B(z^1,R)\cap \Omega_1$ in which the supremum in (\ref{tau}) is realised.
There are two possibilities:

a) $B(z^1,aR)\cap \Gamma_2 = \emptyset$; 

b) $B(z^1,aR)\cap \Gamma_2 \ne \emptyset$. 


Let us start with case (a). By assumption (B),  $B^0 \cap \Omega \cap B^1$ contains a ball of radius $q R$. 
Due to Proposition \ref{Land} and (\ref{tau}) it follows that
\begin{equation}\label{func_v11}
\sup_{B(z^1,aR)\cap \Omega} u_1\ge  \frac {\sup_{B^1\cap \Omega} u_1}{1-\eta_3}\ge \frac{m\zeta_0(1-\tau)}{1-\eta_3},
\end{equation}
where $\eta_3=\eta_1q^s$.
 
Using  (\ref{func_v1}) and \eqref{func_v11} we deduce
\begin{equation*}\label{u-tau1}
 \sup_{B(z^1,aR)\cap \Omega}u \geq m\,\Big(1+\frac {\zeta_0(\eta_3-\tau)}{1-\eta_3}\Big).  
\end{equation*}
Letting $\tau=\frac{\eta_3}2$ we get 
\begin{equation}\label{u-eta4}
M \geq  \sup_{B(\xi_1,aR)\cap \Omega}u \geq m\,\Big(1+\frac {\zeta_0\tau}{1-2\tau}\Big),  
\end{equation}
and the statement follows.

In case of (b), recalling the assumption (C) we proceed with the same arguments but apply Lemmas \ref{barrier-2} and \ref{growth-one} instead of Proposition \ref{Land} 
and put $\tau=\frac{\beta_0}2$. Thus, if the relation (\ref{tau}) holds with $\tau=\frac 12\min\{\eta_3,\beta_0\}$ then (\ref{u-eta4}) is satisfied in any case, and 
Lemma is proved.

If \eqref{tau} does not hold then function $u$ satisfies
\begin{equation*}\label{u-tau}
 \sup_{B^1\cap \Omega}u \leq m(1-\zeta_0\tau). 
\end{equation*}
As in previous step we consider the function
\begin{equation*}\label{func_v2}
u_2(x)=u(x)- m(1-\zeta_0\tau),\qquad u_2(x)\leq 0 \quad \text{in}\ \ \Omega \cap B^1.
\end{equation*}

Repeating previous argument we deduce that if
\begin{equation} \label{v-2}
\sup_{B^2\cap \Omega} u_2 \ge m\zeta_0 \tau(1 -\tau)
\end{equation}
then
$$
M \geq m\big(1+\frac {\zeta_0\tau^2}{1-2\tau}\big),  
$$
and Lemma is proved.

If \eqref{v-2} does not hold, then 
$$
 \sup_{B^2\cap \Omega}u \leq m(1-\zeta_0\tau^2). 
$$
Repeating this process we either prove Lemma or arrive at the inequality
$$
 \sup_{B^N\cap \Omega}u \leq m(1-\zeta_0\tau^N) 
$$
that is impossible since $\xi\in B^N$ and $u(\xi)=m$.
\end{proof}

\begin{remark}
Sometimes it is more convenient to use the following corollary of \eqref{M>m}:
\begin{equation}\label{mod_M>m}
\sup_{ \D(j-1,j+1)} u \ge \sup_{ \D(j)} u \cdot(1 + \kappa_j \mathbf{C}_{s_j}(H_j)R^{-s_j}). 
\end{equation}

\end{remark}

\section{Dichotomy of solutions} \label{dichotomy}

In this section we will apply the Growth Lemma for admissible domain obtained in the previous section to prove dichotomy of solutions at infinity. 
Let domain $\D\subset\tilde \D$ be admissible.  Denote $M(\tau)=\sup\limits_{ x_1=\tau} u$. In particular, $M(\tau_j)=\sup\limits_{ \D(j)} u$.

We start with the following elementary consequence of the maximum principle.

\begin{proposition}\label{mon}
Suppose that $u$ is subject to (\ref{subell}). Then
\begin{itemize}
\item
either there is $\tau^*\ge\tau_1$ s.t.   for $\tau_{j+1}>\tau_j>\tau^*$ we have $M(\tau_{j+1})>M(\tau_j)$;
\item
or for all  $\tau_j>\tau_1$ we have $M(\tau_{j+1})<M(\tau_j)$.
\end{itemize}

\end{proposition}

\begin{theorem} \label{narrow-thm}
Let the condition \eqref{cap-cond} be satisfied, and
\begin{equation}\label{ser_div}
\sum_{j=1}^{\infty} \kappa_j \mathbf{C}_{s_j}(H_j) \, R_j^{-{s_j}} = \infty
\end{equation}
(we recall that $\kappa_j$ is defined in \eqref{kapaj}, $H_j=\tilde\D(j-1,j+1)\setminus\D$ and $s_j= e(j)-2$).  

Then for any function $u$ subject to (\ref{subell}) we have the following dichotomy:

\begin{enumerate}
 \item 
either  $M(\tau) \to \infty $ as $ \tau \to \infty$ and 
$$ 
\liminf_{N \to \infty}\, \frac{M(\tau_N)}{2^{\sum_{j=1}^N \kappa_j\mathbf{C}_{s_j}(H_j)R_j^{-{s_j}} }} \ge c_1;
$$

\item
 or $M(\tau) \to 0$ as $ \tau \to \infty$ and 
 $$
 \limsup_{N \to \infty} M(\tau_N)\cdot  2^{\sum_{j=1}^N \kappa_j \mathbf{C}_{s_j}(H_j)R_j^{-{s_j}} }  \le c_2.
 $$
\end{enumerate}
Here $c_1$ and $c_2$ are some positive constants.
 
\end{theorem}

\begin{proof}
By Lemma \ref{mixed-lm-2}, the relation (\ref{mod_M>m}) is fulfilled for $\tau_j\ge\tau_{j_*}$. For simplicity let $j_*=1$.
By Proposition \ref{mon}, there are two cases.

{\bf Case 1:} there is $j_0 > 1$ s.t.   for all $ j \ge j_0$ we have $M(\tau_{j+1})>M(\tau_j)$.

{\bf Case 2:}  for all $j>1$ we have $M(\tau_{j+1})<M(\tau_j)$.\medskip

In case {\bf 1} we obtain by (\ref{mod_M>m})
$$
\sup_{ \D(j+1)  }u \ge \sup_{ \D(j)  }u \cdot ( 1 + \kappa_j \mathbf{C}_{s_j}(H_j)R_j^{-{s_j}}), \quad  \text{ for }  j \ge j_0.
$$
Thus,
$$ 
\sup_{ \D(N) }u \ge \sup_{ \D(j_0) }u \cdot \prod_{j=j_0+1}^N( 1 + \kappa_j \mathbf{C}_{s_j}(H_j)R_j^{-{s_j}}), \quad { for  } \  N > j_0.
$$
Notice that $\kappa_j \mathbf{C}_{s_j}(H_j)R_j^{-{s_j}}<1$ by (\ref{M>m}). Since $\ln(1+t)>t\ln(2)$ for $t\in(0,1)$, we obtain for $ N > j_0$
$$
\ln M(\tau_N) \ge    \ln M(\tau_{j_0})+ \sum_{j=j_0+1}^N \ln( 1 + \kappa_j \mathbf{C}_{s_j}(H_j)R_j^{-{s_j}}) \ge    
 \ln M(\tau_{j_0})+\ln(2)\sum_{j=j_0+1}^N  \kappa_j \mathbf{C}_{s_j}(H_j)R_j^{-{s_j}},
 $$
or
$$
M(\tau_N) \ge  M(\tau_{j_0}) \cdot 2^{\sum_{j=j_0+1}^N  \kappa_j \mathbf{C}_{s_j}(H_j)R_j^{-s_j} }
\ge c_1\cdot 2^{\sum_{j=1}^N \kappa_j \mathbf{C}_{s_j}(H_j)R_j^{-{s_j}} }.
$$
\medskip

In case 2, we apply (\ref{mod_M>m}) again to get for $N \ge 3$
$$
M(\tau_2) \ge M(\tau_3)\cdot( 1 + \kappa_3 \mathbf{C}_{s_j}(H_3)R_3^{-{s_j}}) \ge \cdots \ge    
M(\tau_N)\cdot\prod_{j=3}^{N}( 1 + \kappa_j \mathbf{C}_{s_j}(H_j)R_j^{-{s_j}}).
$$
Arguing as above, we obtain 
$$ 
M(\tau_2) \ge M(\tau_N)\cdot 2^{\sum_{j=3}^{N} \kappa_j\mathbf{C}_{s_j}(H_j)R_j^{-{s_j}} }
\ge c_1\cdot 2^{\sum_{j=1}^N \kappa_j \mathbf{C}_{s_j}(H_j)R_j^{-{s_j}} }.
$$
This completes the proof of the Theorem.
\end{proof}

To illustrate Theorem \ref{narrow-thm} we consider the domain
\begin{equation}\label{funnel}
\tilde\D=\{x=(x_1, x_1^{\alpha}\bar x  ), \quad |\bar x|<1,\ x_1>1 \},\qquad \alpha<1
\end{equation}
(see Fig. \ref{Funel} for the case $\alpha<0$).  

\begin{figure}[ht]
\begin{center}
\advance\leftskip-3cm
\advance\rightskip-3cm
\includegraphics[scale=0.5]{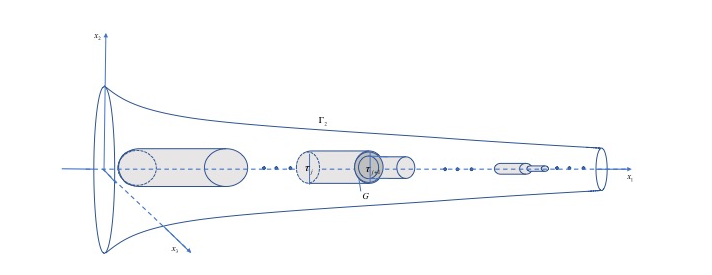}
\caption{ Funnel type domain $\D$. }
\label{Funel}
\end{center}\end{figure}

By Example 1 in Section \ref{growth-admissible}, $\tilde \D$ is admissible w.r.t. exterior normal vector field, if we choose $\tau_j=j^{\frac{1}{1-\alpha}}$.

Let $G$ be the union of cylinders:
$$
G=\bigcup\limits_{j>1} \Big(\overline{B}(0, c\tau_j^{\alpha})\times[\tau_j,\tau_{j+1}]\Big),
$$
where $c$ is sufficiently small positive constant.

Set $\D=\tilde \D\setminus G$, $\Gamma_2=\partial\tilde\D$, $\Gamma_1=\partial\D\setminus\Gamma_2$. It is easy to see that $G\subset\tilde \D$, 
and the domain $\D$ is admissible. Moreover, the assumption (\ref{cap-cond}) is satisfied with $z^0=0$, $R_j=2c(\tau_{j+1}-\tau_j)$ and $\varkappa_j=c$.

From  monotonocity of the $s$-capacity it follows (see chapter 2 in  \cite{landis-book}) that 
$$
\mathbf{C}_{s_j}(\tilde\D(j-1,j+1)\setminus\D) \ge C (\tau_{j+1}-\tau_j)(c\tau_j^{\alpha})^{{s_j}-1} \ge\frac Cc \, R_j^{s_j}.
$$

Taking into account (\ref{kapaj}) we rewrite the assumption (\ref{ser_div}) as follows:
\begin{equation}\label{sum_kappa_lambda}
\frac{C}{2^{N_0}} \sum_{j=1}^{\infty} \lambda^{s_j}=\infty, \qquad \lambda=\frac{2q^{N_0}}{a}.
\end{equation}
We recall that $a$, $q$ and $N_0$ are the constants from Definition \ref{adm} and $s_j= e(j)-2$. Notice that $\lambda<\frac 12$ since $q<\frac a4$.

Theorem \ref{narrow-thm} shows that if (\ref{sum_kappa_lambda}) holds then for any function $u$ subject to (\ref{subell})
one of two possibilities can happen: 
either
\begin{equation}\label{funnel_growth}
M(\tau_N) \ge c_1\cdot 2^{ \tilde c \sum_{j=1}^{N}\lambda^{s_j}},
\end{equation}
 or 
 \begin{equation} \label{funnel_decay} 
 M(\tau_N) \le  c_2\cdot 2^{-\tilde c \sum_{j=1}^{N} \lambda^{s_j}},
 \end{equation}
for some positive constants $c_1$, $c_2$, $\tilde c$.\medskip

{\bf Example 2:} { \bf{Uniformly elliptic equation.}} Let  $\sup_{j}e(j) < \infty$. Then $\lambda^{s_j}\ge const$. Therefore, 
\eqref{funnel_growth} and \eqref{funnel_decay} give the following behavior of sub-elliptic function at infinity (we recall that $\tau_N=N^{\frac 1{1-\alpha}}$): either
$$ 
M(r) \ge c_1 \exp\,({\hat c  r^{1-\alpha}}),
$$
or
$$
M(r) \le c_2 \exp\,({-\hat c  r^{1-\alpha}}).
$$
\smallskip

{\bf Example 3:} {\bf{Degenerate equation.}} Let the ellipticity function grow at infinity as $o(\ln(x_1))$. Namely, we suppose that $s_j=p(j)\ln(j)$, $j>1$, where
 \begin{equation} \label{o-ln} 
p(t)\searrow 0, \qquad p(t)\ln(t)\nearrow \infty, \qquad {\rm as}\quad t\to\infty.
 \end{equation}

We claim that
 \begin{equation} \label{sum-behavior} 
\sum\limits_{j=1}^{N} \lambda^{s_j}\sim N^{1+p(N)\ln(\lambda)} \qquad {\rm as}\quad N\to\infty.
 \end{equation}
Indeed, the series in (\ref{sum-behavior}) evidently diverges, and thus
$$
\sum\limits_{j=1}^{N} \lambda^{s_j}\sim \int\limits_1^N\lambda^{p(t)\ln(t)}dt.
$$
By the L'Hospital rule we derive, as $N\to\infty$, 
$$
\frac {\int_1^N\lambda^{p(t)\ln(t)}dt}{N^{1+p(N)\ln(\lambda)}}\sim \frac {\lambda^{p(N)\ln(N)}}{N^{p(N)\ln(\lambda)}\cdot(1+p(N)\ln(\lambda)+Np'(N)\ln(\lambda))}
=\frac 1{1+p(N)\ln(\lambda)+Np'(N)\ln(\lambda)}.
$$
However, (\ref{o-ln}) implies
$$
0<t(p(t)\ln(t))'=p(t)+tp'(t)<p(t)\searrow 0,\qquad t\to\infty, 
$$
and (\ref{sum-behavior}) follows.

Thus, in this case \eqref{funnel_growth} and \eqref{funnel_decay} give the following behavior of sub-elliptic function at infinity (we use the fact that 
$\ln(\lambda)<0)$: 
either
$$
M(r) \ge c_1 \exp\,\big({\hat c r^{(1-\alpha)(1-\overline c p(r^{1-\alpha}))}}\big),
$$
or 
$$
M(r) \le c_2 \exp\,\big({-\hat c r^{(1-\alpha)(1-\overline c p(r^{1-\alpha}))}}\big).
$$
Obviously, these estimates are worse that in uniformly elliptic case.\bigskip


\textbf{Acknowledgement.} D. Cao and A. Ibraguimov's research is supported by DMS NSF grant 1412796. A. I. Nazarov's research is supported by RFBR grant 18-01-00472.
\bigskip

\bibliographystyle{plain}

\end{document}